\documentclass[12pt]{iopart}
\usepackage{iopams}
\eqnobysec
\newtheorem{theorem}{Theorem}[section]

\newtheorem{lemma}[theorem]{Lemma}

\newenvironment{proof}[1][Proof]{\noindent\textbf{#1.} }{\ \rule{0.5em}{0.5em}}

\begin{document}

 \title[Global existence and blow-up for a two-dimensional nonlocal equation]{The
 Cauchy problem for a class of two-dimensional nonlocal nonlinear wave equations governing
 anti-plane shear motions in  elastic materials}

\author{H A Erbay$^{1}$, S Erbay$^{1}$ and A Erkip$^{2}$}

\address{$^1$ Department of Mathematics, Isik  University, Sile 34980, Istanbul, Turkey}
\ead{erbay@isikun.edu.tr}

\address{$^2$ Faculty of Engineering and Natural Sciences, Sabanci University,
        Tuzla 34956,  Istanbul,    Turkey}

  \begin{abstract}
 This paper is concerned with  the analysis of the Cauchy problem of  a general class of two-dimensional
 nonlinear nonlocal wave equations  governing anti-plane shear motions in  nonlocal elasticity. The
 nonlocal nature of the problem is reflected by a convolution integral in the space variables. The Fourier
 transform of the convolution kernel  is nonnegative and satisfies a certain growth condition at infinity.
 For initial data in $L^{2}$ Sobolev spaces, conditions for global existence or finite time blow-up
 of the solutions of the Cauchy problem are established.
 \end{abstract}

\submitto{NL}
\maketitle

 \section{Introduction}

\noindent
In the present paper we consider the initial value problem
\begin{eqnarray}
    && w_{tt}=\left( \beta \ast {{\partial F}\over {\partial w_{x}}} \right)_x
            +\left( \beta \ast {{\partial F}\over {\partial w_{y}}} \right)_y,
            ~~~(x,y)\in {\Bbb R^2},~~t>0, \label{cau1} \\
    && w(x,y,0)=\varphi(x,y),~~~w_{t}(x,y,0)=\psi(x,y), \label{cau2}
\end{eqnarray}
where (\ref{cau1}) models anti-plane shear motions  in nonlinear nonlocal elasticity, in terms of non-dimensional
quantities.  In (\ref{cau1})-(\ref{cau2}), $w=w(x,y,t)$ represents the out-of-plane displacement,  the strain energy
density function $F$  is a nonlinear function of $| \nabla w|^{2} \equiv (w_x^2+w_y^2)$ for isotropic materials
with $F(0)=0$, and  the subscripts denote partial derivatives. The terms with $\beta$ in (\ref{cau1}) incorporate
the nonlocal effects where
\begin{equation*}
    (\beta \ast u)(x,y)=\int_{\mathbb{R}^2}\beta (x-x',y-y') u(x',y') dx' dy'
\end{equation*}
denotes convolution  of $\beta$ and $u$. The kernel $\beta(x,y)$ is  assumed to be an integrable function whose
Fourier transform, $\widehat{\beta}(\xi_1,\xi_2)$, satisfies
\begin{equation}
    0\le \widehat{\beta}(\xi)\leq C (1+| \xi |^{2})^{-r/2},~~\mbox{for all}~\xi=(\xi_1,\xi_2), \label{beta}
\end{equation}
where  $C$ is a positive constant and $r\geq 2$. The aim of this paper is to establish  the well-posedness of
the initial value problem  (\ref{cau1})-(\ref{cau2}), as well as the global existence and blow-up of
solutions for a wide class of  the kernel functions $\beta(x,y)$.  The number $r$ in (\ref{beta}) is closely related to
the smoothness of $\beta$ and, consequently, as the decay rate $r$ gets larger the regularizing  effect of the nonlocal behavior
increases. This situation is clearly observed through a comparison of Theorem \ref{theo3.7} and Theorem \ref{theo3.8}.

Although the model requires the nonlinearity to be of the isotropic form $F(w_x^2+w_y^2)$, our
results also apply to the more general $F(w_x, w_y)$-type nonlinearities corresponding to the anisotropic case.
The bulk of our work deals with the isotropic form of $F$ but in a separate section we present all the necessary modifications
corresponding to the anisotropic form. Similarly, in the model the kernel $\beta$ is a function of the
modulus $|(x,y)|$, but we do not require this restriction on $\beta$ in our work. Basically, the approach presented
here extends the techniques used for the one-dimensional nonlinear nonlocal Boussinesq-type wave equations in
the previous studies \cite{duruk1, duruk2, duruk3} to the two-dimensional wave equation given by (\ref{cau1}).
It is worthwhile observing that when $\beta$ is taken as  the Dirac measure in (\ref{cau1}), one recovers the
quasilinear wave equation for anti-plane shear motions of the local (classical) theory of elasticity. A natural question
is what happens if (\ref{cau1})-(\ref{cau2}) is considered on a bounded domain. This question  requires a careful
interpretation of the convolution integral and of possible boundary conditions. We refer to the recent manuscript
\cite{andreu}  where such interpretations  are studied for  nonlocal diffusion problems on bounded domains.

The plan of the paper is as follows:  In Section 2 we give a brief formulation of the anti-plane shearing problem of
nonlocal elasticity. In Section 3  we present a local existence theory  for solutions of the Cauchy problem
(\ref{cau1})-(\ref{cau2}) for given initial  data in  suitable Sobolev  spaces. In Section 4 we prove   global
existence of solutions of  (\ref{cau1})-(\ref{cau2}) assuming some  positivity condition on the nonlinear function
$F$ together with  enough smoothness on the initial data. In Section 5 we discuss  finite time blow-up of
solutions. Finally, in Section 6 we show that, with certain modifications, all the results of Sections 2-5 are also
valid for more general $F(w_x, w_y)$-type nonlinearities.

Throughout the paper,
$\widehat{u}(\xi)=\mathcal{F}(u) (\xi)=  \int_{\Bbb R^2} e^{-i z \cdot \xi} u(z) dz$ and
$\mathcal{F}^{-1}(\widehat{u}) (z)=\frac{1}{(2\pi)^2}\int_{\Bbb R^2} e^{i z \cdot \xi}\widehat{u}(\xi) d \xi$
denote the Fourier transform and inverse Fourier transform, respectively, where
$z=(x,y),~\xi=(\xi_1,\xi_2), ~dz=dx dy$, and $d\xi=d\xi_1 d\xi_2$. Furthermore, $H^{s}({\Bbb R^2})$ denotes the
$L^{2}$ Sobolev space on ${\Bbb R^2}$. For the $H^{s}$ norm we use the  Fourier transform representation
 $~\left\Vert u \right\Vert_{s}^{2}=\int_{\Bbb R^2} (1+| \xi|^{2})^{s}|\widehat{u}(\xi)|^{2}
 \mbox{d}\xi~$. Also, $~\left\Vert u \right\Vert_{L^{\infty}}~$ and $~\left\Vert u \right\Vert~$
 indicate the $L^{\infty}$ and $L^{2}$ norms, respectively, and $\langle u, v \rangle$ refers to the inner product of
 $u$ and $v$ in $L^2(\Bbb R^2)$.

\section{The Model}

In this section we discuss  how equation (\ref{cau1}) can be derived to describe the propagation of a
finite amplitude transverse wave  in a nonlocally elastic medium.  Before stating our derivation, we need to introduce
the concept of nonlocal elasticity. One of the major drawbacks in the local theory
of elasticity is that it does not include any intrinsic length scale and consequently does not take into account
the long range forces that become increasingly important at small scales. As a result,  the local theory of
elasticity is incapable of predicting, for instance,  (i) the dispersive nature of harmonic waves in crystal lattices
and (ii) the boundedness of the stress field near the tip of a crack. In order to overcome such deficiencies various
generalizations of the local theory of elasticity have been proposed. One such generalization is the theory of
nonlocal elasticity which has been developed by Kr\"oner \cite{kroner},  Eringen and Edelen \cite{eringen1}, Kunin
\cite{kunin}, Rogula \cite{rogula}, Eringen \cite{eringen2, eringen3}  over the last several decades (For more recent
studies on the subject of generalized theories of elasticity, see, for instance,
\cite{silling, polizzotto, eskandarian, arndt, blanc, huang} and references therein). What distinguishes the theory of
nonlocal elasticity from the local theory of elasticity is that the stress at a point depends on the strain field
at every point in the body. Although there has been a considerable amount of research done on small scale effects within
the context of the theory of nonlocal elasticity,  they are mostly restricted to linear models. Recently, in
\cite{duruk1, duruk2, duruk3} various Cauchy problems based on a one-dimensional nonlinear model of nonlocal elasticity
have been studied. Here, we show how the approach in those studies is extended to the dynamic anti-plane shearing
problem of nonlinear nonlocal elasticity.

Consider an isotropic homogeneous nonlocally elastic medium. Identify a material point $\bf X$ of the medium by its
rectangular Cartesian coordinates in a  reference configuration: ${\bf X}=(X_{1},X_{2},X_{3})$.
We assume that the reference configuration is unstressed. Let
${\bf x}({\bf X},t)=(x_{1}({\bf X},t),x_{2}({\bf X},t),x_{3}({\bf X},t))$ denote the position of the same point at
time $t$. Then the displacement and the deformation gradient are given by ${\bf u}({\bf X},t)={\bf x}({\bf X},t)-{\bf X}$
and ${\bf A}({\bf X},t)=\mbox{Grad}~{\bf x}({\bf X},t)$, respectively. We suppose that a (local) strain energy density
function $F({\bf A})$ per unit volume of the undeformed reference configuration exists, i.e., the material is (locally)
hyperelastic, and that it sustains a nontrivial anti-plane shear motion.   In the local theory of elasticity, a
constitutive equation of the form
$\boldsymbol{\sigma}=\boldsymbol{\sigma}({\bf A})\equiv \partial F({\bf A})/\partial {\bf A}$ holds for a hyperelastic
material (see equation (4.3.7) of \cite{ogden}), where $\boldsymbol{\sigma}$ is the nominal stress tensor (note that some
authors use its transpose referred to as the first Piola-Kirchhoff stress tensor). In the theory of nonlocal  elasticity
the (nonlocal) stress tensor ${\bf S}$ is related to the (local) stress tensor $\boldsymbol{\sigma}$ through the
constitutive relation
${\bf S}={\bf S}({\bf X},t)\equiv \int \beta(|{\bf X}-{\bf Y}|){\boldsymbol{\sigma}}({\bf A}({\bf Y},t))d{\bf Y}$ where
$\beta(|{\bf X}-{\bf Y}|)$ is a kernel function that weights the contribution of the local stresses to the nonlocal
stresses.  In the absence of body forces, the (Lagrangian) equation of motion (see equation (3.4.4) of \cite{ogden}) is
given by $\rho_{0}\ddot{\bf x}=\mbox{Div}~{\bf S}$ where $\rho_{0}$ is the mass density of the medium and a superposed
dot indicates the material time derivative. The only difference between the equations of the local theory of
elasticity and those of the nonlocal model presented here is due to the constitutive equations.

Now we consider an anti-plane shear motion of the form
\begin{equation}
    x_{1}=X_{1},~~~~x_{2}=X_{2},~~~~x_{3}=X_{3}+w(X_{1},X_{2},t) \label{anti}
\end{equation}
for a nonlocally elastic material, where the out-of-plane displacement $w$ is the only non-zero component of
displacement, i.e. $u_{1}=u_{2}\equiv 0$ and $u_{3}\equiv w(X_{1},X_{2},t)$. We henceforth replace the arguments
$X_{1}$ and $X_{2}$ of the displacement $w$ with $x$ and $y$, respectively, and denote partial differentiations with
subscript letters. For isotropic materials the strain energy density  function $F$ is a function of the three
fundamental scalar invariants of the left Cauchy-Green matrix  and for the anti-plane shear motion (\ref{anti})
it turns out to be a function of $w_x^2+w_y^2$ alone: $F=F(w_x^2+w_y^2)$ (see, for instance, Section 4 of \cite{horgan}).
Furthermore,  the equation of motion reduces to the scalar partial differential equation
$\rho_{0}w_{tt}=\left( \beta \ast \sigma_{13} \right)_x+\left( \beta \ast \sigma_{23} \right)_y$ where
$\sigma_{13}$ and $\sigma_{23}$ are the (local) shear stresses arising due to the anti-plane shear motion and they
are given by $\sigma_{13}={\partial F}/{\partial w_{x}}$ and $\sigma_{23}={\partial F}/{\partial w_{y}}$.  The
computations are identical to those in the conventional formulation of nonlinear elasticity, provided we replace the
nonlocal stress tensor with the local stress of conventional theory of elasticity \cite{horgan}. The nonlocal behavior
is represented by the convolution integral. Thus, without loss of generality, if we make a suitable
non-dimensionalization of the equation of motion (see \cite{duruk2} for the non-dimensionalization in the one-dimensional
case) and use the same symbols to avoid a proliferation of notation, or simply take the mass density to be 1, we get
(\ref{cau1}). Equation (\ref{cau1}) is consistent with that of the conventional formulation of nonlinear elasticity. In
other words, when $\beta$ is taken as  the Dirac measure to eliminate the nonlocal effect, (\ref{cau1}) reduces  to the
quasilinear wave equation governing anti-plane shear  motions in the local theory of nonlinear elasticity (see
for instance equation (7.10) of \cite{horgan} or equation (2.2) of \cite{lott}). A list of the most commonly used
one-dimensional kernel functions that satisfy the one-dimensional version of the condition given in (\ref{beta}) is
presented in \cite{duruk2}. We now present three examples of two-dimensional kernel functions used in  the literature.
\begin{itemize}
\item[{\it (i)}] {\it The Gaussian kernel} \cite{eringen4}: $\beta(x,y)=(2\pi)^{-1}e^{-(x^{2}+y^{2})/2}$. We have
        $\widehat{\beta}(\xi_{1},\xi_{2})=e^{-(\xi_{1}^{2}+\xi_{2}^{2})/2}$. This is a highly regularizing
        kernel as can be observed by the fact that we can take any $r$  in (\ref{beta}).
\item[{\it (ii)}] {\it The modified Bessel function kernel} \cite{eringen4}:
        $\beta(x,y)=(2\pi)^{-1}K_{0}(\sqrt{x^{2}+y^{2}})$ where $K_{0}$ is the modified Bessel function of the second
        kind of order zero. Since $\widehat{\beta}(\xi_{1},\xi_{2})=(1+\xi_{1}^{2}+\xi_{2}^{2})^{-1}$, for this special
        case we have $r=2$ in (\ref{beta}). Note that $\beta$ is the Green's function for the operator
        $(1-\Delta)$ where $\Delta$ denotes the two-dimensional Laplacian. In this case (\ref{cau1}) can equivalently be
        written as
        \begin{equation*}
        w_{tt}-\Delta w_{tt}=\left( {{\partial F}\over {\partial w_{x}}} \right)_x
            +\left( {{\partial F}\over {\partial w_{y}}} \right)_y.
        \end{equation*}
        Letting $F(s)={1\over 2}s+G(s)$ we obtain the more familiar form
        \begin{equation*}
        w_{tt}-\Delta w -\Delta w_{tt}=\left( {{\partial G}\over {\partial w_{x}}} \right)_x
            +\left( {{\partial G}\over {\partial w_{y}}} \right)_y.
        \end{equation*}
\item[{\it (iii)}] {\it The bi-Helmholtz type kernel} \cite{maugin}:
        \begin{equation*}
          \beta(x,y)=\frac{1}{2 \pi(c_1^2-c_2^2)}   [ K_0 (\sqrt{x^2+y^2}/c_1) -
                                                      K_0 (\sqrt{x^2+y^2}/c_2) ]
        \end{equation*}
        where $c_1$ and $c_2$ are real and positive constants.     Since
        $\widehat{\beta}(\xi_{1},\xi_{2})= [1+\gamma_1 (\xi_1^2+\xi_2^2)+\gamma_2(\xi_1^2+\xi_2^2)^2  ]^{-1}$
        where $\gamma_1=c_1^2+c_2^2$ and $\gamma_2=c_1^2 c_2^2$  we have $r=4$. As above, $\beta$ is  Green's function
        for the operator $(1-\gamma_1\Delta +\gamma_2 \Delta^2)$. Then  (\ref{cau1}) becomes
        \begin{equation*}
        w_{tt}-\Delta w -\gamma_1 \Delta w_{tt}+\gamma_2 \Delta^2 w_{tt}
            =\left( {{\partial G}\over {\partial w_{x}}} \right)_x
            +\left( {{\partial G}\over {\partial w_{y}}} \right)_y.
        \end{equation*}
\end{itemize}
 In the remainder of this paper we discuss the question of well-posedness of the Cauchy problem (\ref{cau1})-(\ref{cau2}).

 \section{Local Existence and Uniqueness of Solutions}

 In the present section, we prove existence and uniqueness of solutions over a small time interval.
 Local well-posedness is established by converting the initial value problem (\ref{cau1})-(\ref{cau2})
 into a system of Banach space-valued ordinary differential equations.
Thus  (\ref{cau1})-(\ref{cau2}) is formally equivalent to the system
\begin{eqnarray}
    &&  w_t=v, ~~~w(0)=\varphi,    \label{system1} \\
    &&  v_t= Kw, ~~~ v(0)=\psi \label{system2}
    \end{eqnarray}
    where the operator $K$ is defined as
     \begin{equation}
    Kw=\left( \beta \ast {{\partial F}\over {\partial w_{x}}} \right)_x
            +\left( \beta \ast {{\partial F}\over {\partial w_{y}}} \right)_y. \label{k}
     \end{equation}
 The Banach space $X^s$  will be defined as
 \begin{equation*}
    X^s= \{ w\in H^s(\mathbb{R}^2);~w_x, w_y\in L^\infty(\mathbb{R}^2) \},
 \end{equation*}
 endowed with the norm
 \begin{equation}
 \Vert w\Vert_{X^{s}}=  \Vert w\Vert_s + \Vert w_x\Vert_{L^{\infty}} + \Vert w_y\Vert_{L^{\infty}}. \label{norm}
 \end{equation}
The following two lemmas  are useful in the proof.
\begin{lemma}\label{lem3.1} (Sobolev Embedding Theorem)
    If $ ~\displaystyle s>\frac{n}{2}+k$, then \\
    $H^{s}(\mathbb{R}^n) \subset C^{k}(\mathbb{R}^n) \cap L^{\infty} (\mathbb{R}^n)$.
\end{lemma}
In particular when $n=2$ and $s>2$, the embedding in Lemma \ref{lem3.1}  implies the norm estimate
    $\Vert ~|\nabla u|~ \Vert_{L^{\infty}} \leq  C\Vert u \Vert_s$. We refer to Chapter 5 of \cite{adams} for a discussion on the many versions of the Sobolev
    embedding theorem.
\begin{lemma}\label{lem3.2}
    Let $s\geq 0$  and let $u_{1},u_{2}\in H^{s}(\mathbb{R}^n)\cap L^{\infty}(\mathbb{R}^n)$. Then $u_{1}u_{2}\in H^{s}(\mathbb{R}^n)$ and
    \begin{equation*}
    \Vert u_{1}u_{2} \Vert_s \leq C (  \Vert u_{1} \Vert_s \Vert u_{2} \Vert_{L^{\infty}}
              +\Vert u_{1} \Vert_{L^{\infty}} \Vert u_{2} \Vert_s ).
    \end{equation*}
\end{lemma}
Lemma \ref{lem3.2} can be found in \cite{kato} in a more general $L^{p}$-setting (see Lemma X4 of \cite{kato}); we also refer to \cite{taylor} for a general discussion.
The following two lemmas (see Chapter 5 of \cite{runst}) have  been used by many authors (for instance, see Lemmas 1, 2 and 3 of \cite{constantin} or
Lemmas 2.3 and 2.4 of \cite{wang2}) to
control the nonlinear terms.
\begin{lemma}\label{lem3.3}
    Let $s\geq 0,$ $f\in C^{[s]+1}(\mathbb{R})$ with $f(0)=0$. Then for any $u\in H^{s}(\mathbb{R}^n)\cap L^{\infty}(\mathbb{R}^n)$, we have
    $f(u) \in H^{s}(\mathbb{R}^n)\cap L^{\infty}(\mathbb{R}^n)$. Moreover there is some constant $A(M)$ depending on
    $M$ (and $s$) such
    that for all $u\in H^{s}(\mathbb{R}^n)\cap L^{\infty}(\mathbb{R}^n)$ with $\Vert u\Vert_{L^{\infty}}\leq M$
        \begin{equation*}
        \Vert f(u)\Vert_s \leq A(M) \Vert u \Vert_s~.
        \end{equation*}
\end{lemma}
\begin{lemma}\label{lem3.4}
    Let $s\geq 0$, $f\in C^{[s]+1}(\Bbb R)$. Then for any $M>0$ there is some constant $B(M)$ depending on
    $M$ (and $s$) such that for
    all $~u_{1},u_{2}\in H^{s}({\Bbb R}^{n})\cap L^{\infty }({\Bbb R}^{n})$ with
    $~\Vert u_{1} \Vert_{L^{\infty}} \leq M$, $~\Vert u_{2}\Vert_{L^{\infty}} \leq M~$ and
    $~\Vert u_{1} \Vert_s \leq M$, $~\Vert u_{2}\Vert_s\leq M~$ we have
    \begin{equation*}
    \Vert f(u_{1})-f(u_{2})\Vert_s      \leq B(M) \Vert u_{1}-u_{2} \Vert_s~.
    \end{equation*}
\end{lemma}
In our case the nonlinearities are of the form
\begin{equation*}
\!\!\!\!\!\!\!\!\!\!\!\!\!\!\!
{{\partial F}\over {\partial w_{x}}}(| \nabla w|^{2})=2w_{x}F^{\prime}(| \nabla w|^{2}),~~~~~
             {{\partial F}\over {\partial w_{y}}}(| \nabla w|^{2})=2w_{y}F^{\prime}(| \nabla w|^{2})
\end{equation*}
where $F^{\prime}$ denotes the derivative of $F$. It follows from repeated applications of Lemma
\ref{lem3.2} that for the above terms  Lemmas \ref{lem3.3} and \ref{lem3.4} take the following forms:
\begin{lemma}\label{lem3.5}
    Let $s\geq 1$, $F\in C^{[s]+1}(\mathbb{R})$. Then for any $w\in X^s$, we have
    \begin{equation*}
\!\!\!\!\!\!\!\!\!\!\!\!\!\!\!
{{\partial F}\over {\partial w_{x}}}(| \nabla w|^{2})\in H^{s-1}({\Bbb R}^{2}),~~~~~
             {{\partial F}\over {\partial w_{y}}}(| \nabla w|^{2})\in H^{s-1}({\Bbb R}^{2}).
\end{equation*}
     Moreover there is some constant $A(M)$ depending on $M$ (and $s$) such that for
     all $w\in X^s$ with $\Vert ~| \nabla w|~\Vert_{L^{\infty}}\leq M$
        \begin{eqnarray*}
        && \Vert {{\partial F}\over {\partial w_{x}}}(| \nabla w|^{2})\Vert_{s-1} \leq A(M) \Vert w \Vert_{s}, \\
        && \Vert {{\partial F}\over {\partial w_{y}}}(| \nabla w|^{2})\Vert_{s-1} \leq A(M) \Vert w \Vert_{s}~.
        \end{eqnarray*}
\end{lemma}
\begin{lemma}\label{lem3.6}
    Let $s\geq 1$, $F\in C^{[s]+1}(\Bbb R)$. Then for any $M>0$ there is some constant $B(M)$ depending on $M$ (and $s$)
    such that for     all $~w_{1},w_{2}\in X^{s}$ with
    $~\Vert  w_{1}\Vert_{X^{s}} \leq M$, $~\Vert w_{2}\Vert_{X^{s}} \leq M~$ we have
    \begin{eqnarray*}
        && \Vert {{\partial F}\over {\partial w_{x}}}(| \nabla w_{1}|^{2})
             -{{\partial F}\over {\partial w_{x}}}(| \nabla w_{2}|^{2})\Vert_{s-1} \leq B(M) \Vert w_{1}-w_{2} \Vert_{s}, \\
        && \Vert {{\partial F}\over {\partial w_{y}}}(| \nabla w_{1}|^{2})
             -{{\partial F}\over {\partial w_{y}}}(| \nabla w_{2}|^{2})\Vert_{s-1} \leq B(M) \Vert w_{1}-w_{2} \Vert_{s}~.
        \end{eqnarray*}
\end{lemma}
We want to emphasize that, while the constants $B(M)$ in Lemmas \ref{lem3.4} and \ref{lem3.6} are indeed local Lipschitz
constants depending on the $X^{s}$-norms, the constants $A(M)$ in Lemmas \ref{lem3.3} and \ref{lem3.5} depend
only on the $L^{\infty}$-norms. This property of $A(M)$ will be used in the proof of Lemma \ref{lem3.9} which
characterizes the type of blow-up.

  When $s>2$ we have the following local well posedness result.
\begin{theorem}\label{theo3.7}
    Suppose $~s>2$ and the decay rate $r$ in (\ref{beta}) satisfies  $~r\geq 2$. For
    $~\varphi, \psi \in H^s({\Bbb R^2})$, there is some $~T>0~$ such that the Cauchy problem (\ref{cau1})-(\ref{cau2}) is
    well posed with solution $~w=w(x,y,t)~$ in $~C^2([ 0,T],H^s({\Bbb R^2}) )~$.
\end{theorem}
\begin{proof}
    Let $w \in H^s(\mathbb{R}^2)$. For $~s>2$, by the Sobolev Embedding Theorem we have
    $~|\nabla w | \in  L^\infty (\mathbb{R}^2)$. Thus $X^{s}=H^{s}({\Bbb R^2})$ and the norm
    $~\Vert  w\Vert_{X^{s}}$ can be replaced by the equivalent   $H^{s}$ norm $~\Vert  w\Vert_{s}$.
    Since (\ref{cau1})-(\ref{cau2}) is formally equivalent to (\ref{system1})-(\ref{system2}), we will use
    the well-known existence-uniqueness (Picard-Lindel\"of) theorem for Banach space valued  systems of  ordinary
    differential equations (for instance, see Theorem 5.1.1 of \cite{ladas}).
    Obviously, all we need is to show that the operator $K$ of (\ref{k})  is locally Lipschitz on $X^s$. We first show
    that $K$ maps $X^{s}$ into $X^{s}$. We estimate the convolution  as
 \begin{equation*}
    \Vert  \beta \ast u \Vert_s = \Vert (1+| \xi |^2)^{s/2} \widehat{\beta}(\xi) \widehat{u}(\xi)\Vert
     \leq  C \Vert (1+| \xi |^2)^{(s-r)/2} \widehat{u}(\xi)\Vert = C \Vert u\Vert_{s-r},
    \end{equation*}
    where we have used inequality (\ref{beta}).  By Lemma \ref{lem3.5} for
    $\Vert ~| \nabla w|~\Vert_{L^{\infty}}\leq M$
    \begin{eqnarray*}
     \Vert \left (\beta \ast {{\partial F}\over {\partial w_{x}}}\right )_{x}\Vert_s
      &\leq & \Vert \beta \ast {{\partial F}\over {\partial w_{x}}}\Vert_{s+1}
    \leq C \Vert {{\partial F}\over {\partial w_{x}}}\Vert_{s+1-r}  \\
      &\leq & C A(M)  \Vert w \Vert_{s+2-r}  \leq C A(M)  \Vert w \Vert_{s}
    \end{eqnarray*}
     where we have used $r\geq 2$. The same holds for the term $\left (\beta \ast {{\partial F}\over {\partial w_{y}}}\right )_{y}$ and
    \begin{equation}
    \!\!\!\!\!\!\!\!\!\!\!\!\!\!\!\!\!\!\!\!\!\!
        \Vert  Kw \Vert_s \leq  C A(M)\Vert w\Vert_{s+2-r}\leq  C A(M)\Vert w\Vert_{s}. \label{kkk}
    \end{equation}
       Similarly, for $w_{1},w_{2}\in X^{s}$ with $\Vert w_{1} \Vert_{s} \leq M$ and $\Vert w_{2} \Vert_{s} \leq M$,
    by Lemma \ref{lem3.6}
    \begin{eqnarray*}
     \!\!\!\!\!\!\!\!\!\!\!\!\!\!\!\!\!\!
      \Vert \left (\beta \ast {{\partial F}\over{\partial w_{x}}}(| \nabla w_{1}|^{2})\right)_{x}
             -\left(\beta \ast {{\partial F}\over {\partial w_{x}}}(| \nabla w_{2}|^{2})\right)_{x}\Vert_{s}
         &\leq & C B(M)  \Vert w_{1}-w_{2} \Vert_{s+2-r} \\
         \!\!\!\!\!\!\!\!\!\!\!\!\!\!\!\!\!\!
                 &\leq & C B(M)  \Vert w_{1}-w_{2} \Vert_{s}.
    \end{eqnarray*}
     As above, the same holds for the term $\left (\beta \ast {{\partial F}\over {\partial w_{y}}}\right )_{y}$.
     So, $K$ is locally Lipschitz on $X^{s}$ and thus the local well posedness of the Cauchy problem is established.
\end{proof}

When $r>3$ in (\ref{beta}), the extra regularizing effect of $\beta$ allows us to improve the result in Theorem
\ref{theo3.7}  to the case of $s\geq 1$.
\begin{theorem}\label{theo3.8}
    Suppose $~s\geq 1~$ and the decay rate $r$ in (\ref{beta}) satisfies  $~r> 3$. For $~\varphi, \psi \in X^{s}$,
    there is some $~T>0~$ such that the Cauchy problem (\ref{cau1})-(\ref{cau2})
    is well posed with solution $w=w(x,y,t)$ in $~C^2([ 0,T],X^s)~$.
\end{theorem}
\begin{proof}
    Similar to the proof of Theorem \ref{theo3.7} it suffices to show that the map $K$ given in (\ref{k}) is
    locally Lipschitz on $X^{s}$. Recall that
    $\Vert  Kw \Vert_{X^{s}}=\Vert  Kw \Vert_{s}+\Vert  (Kw)_{x} \Vert_{L^{\infty}}+\Vert  (Kw)_{y} \Vert_{L^{\infty}}$.
    The term $\Vert  Kw \Vert_{s}$ can be estimated by $\Vert  w \Vert_{s}$ as above. For $\epsilon=r-3>0$ we have
    \begin{eqnarray}
    &&
        \Vert  (Kw)_{x} \Vert_{L^{\infty}} \leq C \Vert  (Kw)_{x} \Vert_{1+\epsilon}
           \leq C \Vert  Kw \Vert_{2+\epsilon} \leq C \Vert  Kw \Vert_{s+1+\epsilon} \nonumber \\
    && ~~~~~~~~~~~~~~~~ \leq C A(M)  \Vert w \Vert_{s+3+\epsilon-r} = C A(M) \Vert w \Vert_{s}. \label{cam}
    \end{eqnarray}
    where we have used (\ref{kkk}) and the Sobolev Embedding Theorem. The same holds for $(Kw)_{y}$
    and a similar estimate as in the proof of Theorem  \ref{theo3.7} shows that $K$ is locally Lipschitz on $X^{s}$.
\end{proof}

 The solution of (\ref{cau1})-(\ref{cau2}) can be extended to a maximal interval $[ 0,T_{\max})$
 where  finite $T_{\max}$ is characterized by the blow up condition
 \begin{equation*}
    \limsup_{t\rightarrow T_{\max}^{-}}
    \left( \Vert w(t)\Vert_{X^{s}}+\Vert w_{t}(t)\Vert_{X^{s} }\right)=\infty.
 \end{equation*}
 Obviously $T_{\max}=\infty$, i.e. there is a global solution, if and only if for any $T<\infty$
 \begin{equation*}
    \limsup_{t\rightarrow T^{-}}  \left( \Vert w(t)\Vert_{X^{s}}+\Vert w_{t}(t)\Vert_{X^{s}}\right) <\infty.
 \end{equation*}
The lemma below characterizes the type of blow-up; namely blow-up occurs in the $L^{\infty}$-norm of $| \nabla w|$.
 \begin{lemma}\label{lem3.9}
    Suppose  that the conditions of Theorem \ref{theo3.7} or Theorem \ref{theo3.8} hold.
    Then $T_{\max}=\infty$,  i.e. there is a global solution of the Cauchy problem (\ref{cau1})-(\ref{cau2}), if and only if for any $T>0$
    \begin{equation*}
        \limsup_{t\rightarrow T^{-}} (\Vert w_{x}(t)\Vert_{L^{\infty}} +\Vert w_{y}(t)\Vert_{L^{\infty}})
        < \infty.
    \end{equation*}
 \end{lemma}
 \begin{proof}
    Since
    \begin{equation*}
        \Vert w_{x}(t)\Vert_{L^{\infty}} +\Vert w_{y}(t)\Vert_{L^{\infty}} \leq  \Vert w(t)\Vert_{X^{s}},
    \end{equation*}
    it suffices to prove that if the solution  exists for $t\in [ 0,T)$ and
     $\Vert w_{x}(t)\Vert_{L^{\infty}} +\Vert w_{y}(t)\Vert_{L^{\infty}} \leq M$ for all $0\leq t < T$ then both
    $\Vert w(t)\Vert_{X^{s}}$ and $\Vert w_{t}(t)\Vert_{X^{s}}$ stay bounded. Integrating equation
    (\ref{cau1}) twice and calculating the resulting double integral as an    iterated integral,  we obtain
    \begin{eqnarray}
    w(t) &=& \varphi +t\psi+\int_0^t ( t-\tau)  (Kw)(\tau)d\tau, \label{wint} \\
    w_t(t) &=& \psi+\int_0^t  (Kw)(\tau)d\tau.  \label{wtint}
    \end{eqnarray}
    But, by (\ref{kkk}),
    $\Vert ( Kw)(\tau)\Vert_s\leq C A(M) \Vert w(\tau)\Vert_{s+2-r}\leq C A(M) \Vert w(\tau)\Vert_s$ where the constant
    $A(M)$ depends only on $M$.  Hence
    \begin{equation*}
    \!\!\!\!\!\!\!\!\!\!\!\!\!\!\!\!\!\!\!\!\!\!\!\!\!\!\!\!\!\!\!\!\!\!\!\!
        \Vert w(t)\Vert_s + \Vert w_t(t)\Vert_s \leq \Vert \varphi \Vert_s
            +(1+T)\Vert \psi \Vert_s +(1+ T) C A(M)
            \int_0^t \Vert w(\tau)\Vert_s  d\tau ,
    \end{equation*}
    and Gronwall's Lemma gives
    \begin{equation}
        \Vert w (t)\Vert_s
        +\Vert w_t(t)\Vert_s \leq (\Vert \varphi \Vert_s
        +(1+T)\Vert \psi \Vert_s) e^{( 1+T)C A(M) T}  \label{blowup}
    \end{equation}
     for all $t\in \left[ 0,T\right)$. We now estimate $\Vert w_{tx}(t) \Vert_{L^{\infty}}$. The estimate for
      $\Vert w_{ty}(t) \Vert_{L^{\infty}}$ follows similarly.  In the case of Theorem \ref{theo3.7} (where $s>2$),
     by the Sobolev Embedding Theorem,
    \begin{equation*}
    \Vert w_{tx}(t) \Vert_{L^{\infty}} \leq C \Vert w_{t}(t) \Vert_{s}
    \end{equation*}
    so that (\ref{blowup}) applies.  In the case of Theorem \ref{theo3.8} (where $r>3$), from
    (\ref{wtint})
    \begin{equation}
    \Vert w_{tx}(t) \Vert_{L^{\infty}} \leq  \Vert \psi_{x} \Vert_{L^{\infty}}
              + \int_{0}^{t}\Vert (Kw)_{x}(\tau) \Vert_{L^{\infty}} ~d\tau . \label{tx}
    \end{equation}
    By (\ref{cam})
    \begin{equation*}
    \Vert (Kw)_{x}(\tau) \Vert_{L^{\infty}} \leq C A(M)  \Vert w(\tau) \Vert_{s},
    \end{equation*}
    and by (\ref{blowup})
    \begin{equation*}
        \Vert w (\tau)\Vert_s \leq (\Vert \varphi \Vert_s +(1+T)\Vert \psi \Vert_s) e^{( 1+T)C A(M) T}.
    \end{equation*}
    Finally, plugging into (\ref{tx}) we obtain the required estimate
    \begin{equation*}
    \Vert w_{tx}(t) \Vert_{L^{\infty}} \leq  \Vert \psi_{x} \Vert_{L^{\infty}}
              + T C A(M) (\Vert \varphi \Vert_s +(1+T)\Vert \psi \Vert_s) e^{( 1+T)C A(M) T},
    \end{equation*}
    which corresponds to the case $s\ge 1$, $r >3$.
 \end{proof}

 \section{Conservation of Energy and Global Existence}

 In the present section we will prove that locally well defined solutions can be extended to the entire time.

 In the study of global existence of solutions the conservation of energy plays a key role.  First, time
 invariance of the energy functional will be shown. To this end,  we  define  an unbounded linear operator $R$
 on $H^{s}({\Bbb R^2})$  as
 $~R u=\mathcal{F}^{-1} \left ( (\widehat{\beta}(\xi))^{-\frac{1}{2}} \widehat{u}(\xi) \right )$
 where $\mathcal{F}^{-1}$ denotes  the inverse Fourier transform and  $ \widehat{\beta}(\xi )$ is defined in (\ref{beta}).
 Obviously,
 $~R^p u=\mathcal{F}^{-1} \left ( (\widehat{\beta}(\xi))^{-\frac{p}{2}} \widehat{u}(\xi) \right )$ for a real number $p$ with
 $u$ in the domain of $~R^p$. On the other hand,
  $R^{-2}u= \beta \ast u$ and formally (\ref{cau1}) can be rewritten as
 \begin{equation}
    R^2w_{tt}= \left( {{\partial F}\over {\partial w_{x}}}\right)_{x}
             +\left( {{\partial F}\over {\partial w_{y}}}\right)_{y}. \label{cau13}
 \end{equation}
 Here we have used the fact that convolution commutes with differentiation in the distribution sense, i.e.
 $(\beta \ast u)_{x}=\beta \ast u_{x}$.
\begin{lemma}\label{lem4.1}
    Suppose that the conditions of Theorem \ref{theo3.7} or Theorem \ref{theo3.8} hold and the solution of the Cauchy problem
    (\ref{cau1})-(\ref{cau2}) exists in $C^{2}(\left[ 0,T\right), X^{s})$.
    If $R\psi \in L^2({\Bbb R^2})$, then   $Rw_t(t) \in L^2({\Bbb R^2})$ for all $t\in [0,T)$. Moreover, if
    $R\varphi\in L^2({\Bbb R^2})$, then  $Rw(t)\in L^2({\Bbb R^2})$ for all $t\in [0,T)$.
\end{lemma}
\begin{proof}
    Formally, from  (\ref{wtint}) we have
    \begin{eqnarray}
        Rw_t(t)= R\psi+\int_0^t  (RKw)(\tau) d\tau ~.  \label{rw}
    \end{eqnarray}
    Note that
    \begin{equation*}
    RKw= \left( \alpha \ast {{\partial F}\over {\partial w_{x}}}\right)_{x}
             +\left(\alpha \ast {{\partial F}\over {\partial w_{y}}}\right)_{y},
   \end{equation*}
   where $\widehat{\alpha}(\xi)=(\widehat{\beta}(\xi))^{1/2}$. Then similar to the derivation of (\ref{kkk})
   (replacing $\beta$ by $\alpha$ and hence $r$ by $r/2$) we get
     \begin{equation*}
        \Vert (RKw)(\tau)\Vert_{s+\frac{r}{2}-2}        \leq C \Vert  w(\tau) \Vert_s.
    \end{equation*}
    Since either ($s>2$ and $r\geq 2$) or ($s\geq 1$ and $r>3$), in both cases we have $s+{r\over 2}-2 > 0$. Thus the
    right-hand side of (\ref{rw})  belongs to $ L^{2}({\Bbb R^2})$ and the  conclusion follows.     The second
    statement follows  similarly from (\ref{wint}).
\end{proof}
\begin{lemma}\label{lem4.2}
    Suppose that the solution of the Cauchy problem (\ref{cau1})-(\ref{cau2}) exists on some interval
    $[ 0,T )$. If $R\psi \in L^{2}({\Bbb R^2})$ and the function $F(| \nabla \varphi|^{2})$ belongs to $L^{1}({\Bbb R^2})$,
    then for any $t\in [0,T)$ the energy
    \begin{equation}
        E( t)={1\over 2}\Vert R w_t( t)\Vert^2 + \int_{\Bbb R^2} F(| \nabla w(t)|^{2}) dxdy \label{energy}
    \end{equation}
    is constant in $[ 0,T)$.
\end{lemma}
\begin{proof}
    By Lemma \ref{lem4.1}, $Rw_t(t) \in L^2(\mathbb{R}^2)$.  Multiplying  (\ref{cau13}) by $w_t$,
    integrating  in $x$ and $y$, and  using Parseval's identity we get
    \begin{eqnarray*}
    && 0=\frac{d}{dt}\int_{\Bbb R^2} {1\over 2}(\widehat{\beta}(\xi ))^{-1} |\widehat{w}_t(\xi,t)|^2 d\xi
      +\int_{\Bbb R^2} F'(| \nabla w(t)|^{2})\frac{\partial}{\partial t}(| \nabla w(t)|^{2}) dx dy \\
        && ~=\frac{d}{dt} \int_{\Bbb R^2} \left ( {1\over 2} (Rw_t(t))^2 + F(| \nabla w(t)|^{2}) \right ) dx dy
    \end{eqnarray*}
    which implies  the conservation of energy.
\end{proof}

 The main result of this section is the following theorem.
\begin{theorem}\label{theo4.3}
    Suppose $s\geq 1$ and the decay rate $r$ in (\ref{beta}) satisfies  $~r>4$. Let $\varphi, \psi \in X^s$,
    $R\psi \in L^2({\Bbb R^2}) $ and $~F(| \nabla \varphi|^{2}) \in L^1({\Bbb R^2})$. If there is some $k>0$ so that
    $F(u)\geq -k u$ for all $u \geq 0$,  then the Cauchy problem (\ref{cau1})-(\ref{cau2}) has a global
    solution in $C^{2}([ 0,\infty), X^s)$.
\end{theorem}
\begin{proof}
    By  Theorem \ref{theo3.8} the Cauchy problem is locally well-posed. Assume $w\in C^{2}([ 0,T), X^s)$  for
    some $T>0$.  Since $F(u)\geq -k u$, for all $t\in [0,T)$ we have
    \begin{eqnarray}
        \Vert R w_t( t)\Vert^2 &=&  2E(0)- 2\int_{\Bbb R^2} F(| \nabla w(t)|^{2}) dx dy, \nonumber \\
        &\leq & 2 E(0)+2k \int_{\Bbb R^2} | \nabla w(t)|^{2} dx dy, \label{ea}
    \end{eqnarray}
    where $E(0)$ is the initial energy. On the other hand  we have
    \begin{eqnarray}
        \Vert R w_t( t)\Vert^2
        &=& \int_{\Bbb R^2}  (\widehat{\beta}(\xi) )^{-1} | \widehat{w}_{t}(\xi,t) |^2 d\xi, \nonumber \\
        &\geq & C^{-1} \int_{\Bbb R^2} (1+| \xi |^2)^{\frac{r}{2}} | \widehat{w}_{t}(\xi,t) |^2 d\xi,\nonumber \\
        &=& C^{-1} \Vert w_t( t)\Vert_{\frac{r}{2}}^2, \label{pa}
    \end{eqnarray}
    where (\ref{beta})  is used.  Combining  (\ref{ea}) and (\ref{pa})
    \begin{eqnarray*}
        \frac{d}{dt}\Vert w(t)\Vert_{\frac{r}{2}}^2
        &\leq & 2 \Vert w(t)\Vert_{\frac{r}{2}} \Vert w_t(t)\Vert_{\frac{r}{2}}\nonumber \\
        &\leq &   \Vert w(t)\Vert_{\frac{r}{2}}^2+ \Vert w_t(t)\Vert_{\frac{r}{2}}^2\nonumber \\
        &\leq &   \Vert w(t)\Vert_{\frac{r}{2}}^2+ C \Vert Rw_t(t)\Vert^2\nonumber \\
        &\leq &   \Vert w(t)\Vert_{\frac{r}{2}}^2+ 2C ( E(0)+k \Vert |\nabla w (t)| \Vert^2 ) \nonumber \\
        &\leq &  2CE(0)+(1+2Ck) \Vert w(t)\Vert_{\frac{r}{2}}^2,
    \end{eqnarray*}
    where $\Vert |\nabla w (t)| \Vert \leq  \Vert w(t)\Vert_{1} \leq  \Vert w(t)\Vert_{\frac{r}{2}}$  is used.
    Gronwall's lemma implies that     $\left\Vert w\left( t\right) \right\Vert_{\frac{r}{2}} $
    stays bounded in $\left[ 0,T\right)$. As $r>4$ we have  $\frac{r}{2}-1>1$ and
    the Sobolev Embedding Theorem implies
    \begin{equation*}
    \Vert |\nabla w (t)| \Vert_{L^{\infty}}\leq \Vert |\nabla w (t)| \Vert_{{\frac{r}{2}}-1}\leq \Vert w(t)\Vert_{{\frac{r}{2}}}.
    \end{equation*}
    We conclude that
    $\left\Vert |\nabla w (t)| (t)\right\Vert_{L^{\infty} }$ also stays bounded in $\left[ 0,T\right)$. By Lemma
    \ref{lem3.9},     this implies a global solution.
\end{proof}

 \section{Blow up }

 In this section a blow-up result for (\ref{cau1})-(\ref{cau2}) will be presented. The following lemma, based on the idea of Levine \cite{levine},
 will be used to prove blow up of solutions in finite time for certain nonlinearities and initial data.
 \begin{lemma}\label{lem5.1}
    Suppose that $~{\cal H}(t)$, $t\geq 0$, is a positive, twice differentiable function satisfying
    ${\cal H}^{\prime \prime }{\cal H}-( 1+\nu) ({\cal H}^{\prime })^2 \geq 0$ where $\nu >0$.
    If ${\cal H}(0) >0$ and ${\cal H}^{\prime }(0) >0$, then
    ${\cal H}(t) \rightarrow \infty $ as $t\rightarrow t_1$ for some
    $t_1\leq {\cal H}(0) /\nu {\cal H}^{\prime }(0) $.
\end{lemma}
\begin{theorem}\label{theo5.2}
     Suppose that the solution, $w$, of the Cauchy problem (\ref{cau1})-(\ref{cau2}) exists,
     $R\varphi,~ R\psi \in L^2(\Bbb R^2)$ and $~F(| \nabla \varphi|^{2}) \in L^1({\Bbb R^2})$. If there exists a positive number
     $\nu$ such that
    \begin{equation}
       u F^\prime(u) \leq  (1+2\nu) F(u) ~~\mbox{ for all }u \geq 0, \label{blow-condition}
    \end{equation}
    and
    \begin{equation*}
        E(0)= {1\over 2} \Vert R\psi \Vert^2 +\int_{\Bbb R^2} F( | \nabla \varphi|^{2}) dx dy<0,
    \end{equation*}
    then the solution, $w$, of the Cauchy problem (\ref{cau1})-(\ref{cau2}) blows up in finite time.
\end{theorem}
\begin{proof}
    We assume that the global solution to (\ref{cau1})-(\ref{cau2}) exists. Then, by Lemma \ref{lem4.1},
    $Rw(t), ~Rw_t(t) \in L^2({\Bbb R^2})$   for all $t>0$. Let ${\cal H}(t)=  \Vert Rw (t) \Vert^2+b( t+t_0)^2$ where
    $b$ and $t_0$ are positive constants to be  determined later. Then we have
    \begin{eqnarray*}
        {\cal H}^\prime         &=& 2\langle Rw_t, Rw \rangle +2b( t+t_0)  \\
       {\cal H}^{\prime\prime}  &=& 2 \Vert Rw_t \Vert^2
            +2 \langle  Rw_{tt}, Rw \rangle +2b.
    \end{eqnarray*}
    Note that ${\cal H}^{\prime }(0) =2\langle R\varphi ,R\psi \rangle+2bt_0 > 0$
    for sufficiently large $t_0$.  Using  the inequality $u F^\prime(u) \leq  (1+2\nu) F(u)$ together with
    (\ref{ea}) and  (\ref{pa}) we have
    \begin{eqnarray*}
         \langle Rw_{tt}, Rw  \rangle &=&  \langle R^2w_{tt},w \rangle \\
                &=& -2\int_{\Bbb R^2} | \nabla w|^{2} F^\prime(| \nabla w|^{2}) dx dy \\
        &\geq &-2 (1+2 \nu) \int_{\Bbb R^2} F(| \nabla w|^{2}) dx dy\\
        &=& (1+2\nu) \left( \Vert Rw_t\Vert^2-2E(0) \right ),
    \end{eqnarray*}
    so that
    \begin{equation*}
        {\cal H}^{\prime \prime }
            \geq 4( 1+\nu)  \Vert Rw_t \Vert^2-4( 1+2\nu )E(0)+2b.
    \end{equation*}%
    On the other hand, using $2ab \leq a^2+b^2$ and Cauchy-Schwarz inequalities we have
    \begin{eqnarray*}
    \!\!\!\!\!\!\!\!\!\!\!\!\!\!\!\!\!\!\!\!\!\!\!\!\!\!\!\!
        \left({\cal H}^{\prime }\right )^2  &=& 4\left [\langle Rw, Rw_t \rangle +b(t+t_0) \right]^2  \\
    \!\!\!\!\!\!\!\!\!\!\!\!\!\!\!\!\!\!\!\!\!\!\!\!\!\!\!\!
                &\leq & 4 \left( \Vert Rw \Vert^2 ~\Vert Rw_t\Vert^2
                +2b(t+t_0) \Vert Rw \Vert ~\Vert Rw_t \Vert  +b^2 (t+t_0)^2 \right ) \\
    \!\!\!\!\!\!\!\!\!\!\!\!\!\!\!\!\!\!\!\!\!\!\!\!\!\!\!\!
                &\leq & 4 \left( \Vert Rw\Vert^2~\Vert Rw_t\Vert^2
                +b\Vert Rw\Vert^2 +b\Vert Rw_t\Vert^2 (t+t_0)^{2}
                +b^2 (t+t_0)^2 \right ).
    \end{eqnarray*}
    Thus
    \begin{eqnarray*}
    && \!\!\!\!\!\!\!\!\!\!\!\!\!\!\!\!\!\!\!\!\!\!\!\!\!\!
     {\cal H}^{\prime \prime }  {\cal H} - ( 1+\nu)({\cal H}^{\prime })^2 \\
    && \geq  \left (4( 1+\nu)  \Vert Rw_t \Vert^2 -4 (1+2 \nu)  E(0)+2b \right)
                \left (  \Vert Rw \Vert^2 + b(t+t_0)^2  \right )\\
    && -4(1+\nu) \left( \Vert Rw\Vert^2~\Vert Rw_t\Vert^2
            +b \Vert Rw \Vert^2+b \Vert Rw_t\Vert^2 (t+t_0)^2+b^2(t+t_0)^2 \right )\\
    && =-2(1+2\nu) (b+2E(0)) {\cal H}.
    \end{eqnarray*}
    Now if we choose $b\leq -2E(0)$, this gives
    \begin{equation*}
        {\cal H}^{\prime \prime }(t) {\cal H}(t)-(1+\nu)\left ({\cal H}^{\prime }(t) \right)^2\geq 0.
    \end{equation*}
     According to the Blow-up Lemma \ref{lem5.1}, this implies
    that ${\cal H}(t)$, and thus $\Vert Rw(t)\Vert^{2}$ blows up in finite time.
\end{proof}

Consider a typical nonlinearity of the form $F(u)=a u^{q}$ with $q>0$. When $a>0$, Theorem \ref{theo4.3} will apply and
a global solution exists (for suitable $s$, $r$ and initial data). When  $a<0$, the blow-up condition
(\ref{blow-condition}) of Theorem \ref{theo5.2} holds if and only if $q>1$. This observation says that the global existence
result of Theorem \ref{theo4.3} is essentially sharp compared to the blow-up result.

Finally, we  conclude with a short discussion on the condition $E(0)<0$. If $F(u)\geq 0$ for all $u\geq 0$,
then by Theorem \ref{theo4.3}, there is  a global solution. If $F(u)$ is negative on some interval $I$,
we can choose $ \varphi$ with support  in $I$ so that $\int_{\Bbb R^2} F(| \nabla \varphi|^{2}) dx dy<0 $.
Hence, when $\psi= 0$ or $R\psi$ is sufficiently small we get $E(0)<0$. This also shows that blow up may occur even for
 small initial data.

 \section{The Anisotropic Case}

In this section we will consider the more general nonlinear term of the form $F(w_{x}, w_{y}) $ with
$F(0,0)=0$ rather than the isotropic form $F(|\nabla w|^{2})$. Such a form appears as the strain energy function of anisotropic
materials \cite{horgan}. With minor modifications on the assumptions, all the results of Sections 2-5
can be generalized. Following the layout of the manuscript we will  pinpoint out these modifications and
briefly explain how the proofs will change accordingly if we replace $F( |\nabla w|^{2}) $  by
$\tilde{F}(w_{x}, w_{y})=\tilde{F}( \nabla w)$ where the symbol $~\tilde{}~$ is employed to distinguish the anisotropic form of the strain
energy function.

\subsection*{Local Existence and Uniqueness of Solutions}
The main step in Section 2 is to show that the map $K$ of (\ref{k}) is locally Lipschitz on the Banach  space $X^{s}$.
To deal with the more general nonlinearity $\tilde{F}( \nabla w) $ we will use the vector versions of
Lemmas \ref{lem3.3} and \ref{lem3.4} \cite{runst}. For a vector function $U=(u_{1},u_{2})$ the notation
$~\left\Vert U\right\Vert =\left\Vert u_{1}\right\Vert +\left\Vert u_{2}\right\Vert ~$ will be employed where
$\left\Vert {.}\right\Vert$ denotes a given norm.
\begin{lemma}\label{lem6.1}
Let $s\geq 0,$ $h\in C^{[s]+1}(\mathbb{R}^{2})$ with $h(0)=0$. Then for any
$U=(u_{1},u_{2})\in (H^{s}(\mathbb{R}^{2})\cap L^{\infty }(\mathbb{R}^{2}))^{2}$, we have
$h(U)\in H^{s}(\mathbb{R}^{2})\cap L^{\infty}(\mathbb{R}^{2})$. Moreover there is some constant $A(M)$ depending on $M$
such that for all $U\in (H^{s}(\mathbb{R}^{2})\cap L^{\infty }(\mathbb{R}^{2}))^{2}$ with $\left\Vert U\right\Vert_{L^{\infty}}\leq M$
\begin{equation*}
    {\left\Vert h(U)\right\Vert }_{s}\leq A(M){\left\Vert U\right\Vert }_{s}~.
\end{equation*}
\end{lemma}
\begin{lemma}\label{lem6.2}
Let $s\geq 0$, $h\in C^{[s]+1}(\mathbb{R}^{2})$. Then for any $M>0$ there is some constant $B(M)$ such that for all
$~U_{1},U_{2}\in (H^{s}(\mathbb{R}^{2})\cap L^{\infty }(\mathbb{R}^{2}))^{2}$ with
$~\left\Vert U_{1}\right\Vert_{L^{\infty}}\leq M$, $~\left\Vert U_{2}\right\Vert_{L^{\infty}}\leq M~$ and
$~\left\Vert U_{1}\right\Vert_{s}\leq M$, $~\left\Vert U_{2}\right\Vert _{s}\leq M~$ we have
\begin{equation*}
\!\!\!\!\!\!\!\!\!\!\!\!\!\!\!\!\!\!\!\!\!\!\!\!\!\!{\left\Vert
h(U_{1})-h(U_{2})\right\Vert }_{s}\leq B(M){\left\Vert
U_{1}-U_{2}\right\Vert }_{s}~.
\end{equation*}
\end{lemma}
Through Lemmas \ref{lem6.1} and \ref{lem6.2}  we obtain the estimates of Lemmas \ref{lem3.5} and \ref{lem3.6}
for $\tilde{F}( \nabla w) $.
With these new estimates, the proofs of Theorems  \ref{theo3.7} and \ref{theo3.8} on local well posedness and of
Lemma \ref{lem3.9} about the blow up criterion follow exactly the same way for the anisotropic form.

\subsection*{Conservation of Energy and Global Existence}
Due to the new estimates on the map $K$ of (\ref{k}), Lemma \ref{lem4.1} holds also
for the general case. In Lemma \ref{lem4.2} the energy must be replaced by
\begin{equation*}
E(t)={\frac{1}{2}}\Vert Rw_{t}(t)\Vert ^{2}+\int_{\mathbb{R}^{2}}\tilde{F}(\nabla
w(t))dxdy .
\end{equation*}
The proof of the energy identity  is the same with the obvious modification
\begin{eqnarray*}
\frac{d}{dt}\int_{\mathbb{R}^{2}}\tilde{F}(\nabla w(t))dxdy
    &=&\int_{\mathbb{R}^{2}}\left( {\frac{{\partial \tilde{F}}}{{\partial w_{x}}}w_{xt}}
        +{\frac{{\partial \tilde{F}}}{{\partial w_{y}}}}{w_{yt}}\right)dxdy \\
    &=&-\int_{\mathbb{R}^{2}}\left[\left( {\frac{{\partial \tilde{F}}}{{\partial w_{x}}}}\right)_{x}
     +\left( {\frac{{\partial \tilde{F}}}{{\partial w_{y}}}}\right)_{y}\right]{w_{t}}dxdy.
\end{eqnarray*}
In Theorem \ref{theo4.3}, we change the assumption on the nonlinearity as $%
\tilde{F}(U)\geq -k|U|^{2}$ for all $U\in \mathbb{R}^{2}.$ Then the estimate (\ref%
{ea}) still holds for the anisotropic case and thus the rest of the proof follows.
We note that  our new assumption $\tilde{F}(U)\geq -k|U|^{2}$ reduces to  the previous condition $%
F(u)\geq -ku$ of Theorem \ref{theo4.3} when the nonlinearity is of the form $F=F(|\nabla w|^{2})$.

\subsection*{Blow up}
The blow-up condition
\begin{equation*}
uF^{\prime }(u)\leq (1+2\nu )F(u)~~\mbox{ for all }u\geq 0,
\end{equation*}%
of Theorem \ref{theo5.2} takes the form
\begin{equation*}
U\cdot \nabla \tilde{F}(U)\leq 2(1+2\nu )\tilde{F}(U)~~\mbox{ for all }U\in \mathbb{R}^{2}
\end{equation*}
for the general nonlinearity  $\tilde{F}(\nabla w)$.
With this new blow-up condition all the steps in the proof of Theorem \ref{theo5.2} are still valid for the general case
except for the  modification below  in the estimate for the term $\langle Rw_{tt},Rw\rangle $:
\begin{eqnarray*}
         \langle Rw_{tt}, Rw  \rangle &=&  \langle R^2w_{tt},w \rangle \\
                &=& -\int_{\Bbb R^2}  \nabla w \cdot \nabla \tilde{F}( \nabla w) dx dy \\
        &\geq &-2 (1+2 \nu) \int_{\Bbb R^2} \tilde{F}( \nabla w) dx dy\\
        &=& (1+2\nu) \left( \Vert Rw_t\Vert^2-2E(0) \right ).
    \end{eqnarray*}
We note that the new blow-up condition reduces to the blow-up condition of Theorem \ref{theo5.2} for the isotropic form
$F(|\nabla w|^{2})$.
\vspace*{20pt}

\noindent
{\bf Acknowledgement}: This work has been supported by the Scientific and Technological Research Council of Turkey
(TUBITAK) under the project TBAG-110R002. The authors are grateful to the anonymous referees for the insightful comments
and suggestions.

 \end{document}